\documentclass[a4paper,11pt,reqno]{amsart}

\usepackage[utf8x]{inputenc}
\usepackage[T1]{fontenc}
\usepackage{hyperref}
\usepackage{xcolor}
\hypersetup{
    colorlinks,
    linkcolor={red!50!black},
    citecolor={blue!50!black},
    urlcolor={blue!80!black}
}

\usepackage{microtype}
\usepackage{MnSymbol}
\usepackage{textcomp}

\usepackage[mathcal]{euscript}
\let\mathcaltmp\mathcal
\usepackage{mathrsfs}
\let\mathcal\mathscr
\let\mathscr\mathcaltmp

\makeatletter
\def\thm@space@setup{\thm@preskip=7pt
\thm@postskip=7pt}
\makeatother
\newtheoremstyle{plain}
  {}
  {}
  {\slshape}
  {}
  {\bfseries}
  {.}
  { }
  {}

\newtheoremstyle{definition}
  {}
  {}
  {}
  {}
  {\bfseries}
  {.}
  { }
  {}

\makeatletter
\renewenvironment{proof}[1][\proofname]{\par
  \pushQED{\qed}%
  \normalfont \topsep0\p@\relax
  \trivlist
  \item[\hskip\labelsep\itshape
  #1\@addpunct{.}]\ignorespaces
}{%
  \popQED\endtrivlist\@endpefalse
}
\makeatother

\makeatletter
\newcommand{\eqnum}{\refstepcounter{equation}\textup{\tagform@{\theequation}}}
\makeatother

\makeatletter \@addtoreset{equation}{section} \makeatother
\renewcommand{\theequation}{\thesection.\arabic{equation}}
\newtheorem{thm}[equation]{Theorem}
\newtheorem*{thm*}{Theorem}
\newtheorem{lem}[equation]{Lemma}
\newtheorem{cor}[equation]{Corollary}
\newtheorem{prop}[equation]{Proposition}
\newtheorem*{defthm*}{Definition/Theorem}

\theoremstyle{definition}

\newtheorem{rem}[equation]{Remark}

\newtheorem*{exam*}{Example}

\newcommand\arXiv[1]{\href{http://arxiv.org/abs/#1}{arXiv:#1}}

\usepackage{xparse}

\usepackage{xspace}

\newcommand{\changelocaltocdepth}[1]{%
  \addtocontents{toc}{\protect\setcounter{tocdepth}{#1}}%
  \setcounter{tocdepth}{#1}}

\newcommand{\nc}{\newcommand}
\nc{\renc}{\renewcommand}
\nc{\ssec}{\subsection}
\nc{\sssec}{\subsubsection}
\nc{\on}{\operatorname}
\nc{\term}[1]{#1\xspace}

\usepackage{tikz}
\usetikzlibrary{matrix}
\usepackage{tikz-cd}

\tikzset{
  commutative diagrams/.cd,
  arrow style=tikz,
  diagrams={>=latex}}
\makeatletter
\tikzset{
  column sep/.code=\def\pgfmatrixcolumnsep{\pgf@matrix@xscale*(#1)},
  row sep/.code   =\def\pgfmatrixrowsep{\pgf@matrix@yscale*(#1)},
  matrix xscale/.code=%
    \pgfmathsetmacro\pgf@matrix@xscale{\pgf@matrix@xscale*(#1)},
  matrix yscale/.code=%
    \pgfmathsetmacro\pgf@matrix@yscale{\pgf@matrix@yscale*(#1)},
  matrix scale/.style={/tikz/matrix xscale={#1},/tikz/matrix yscale={#1}}}
\def\pgf@matrix@xscale{1}
\def\pgf@matrix@yscale{1}
\makeatother

\usepackage[inline]{enumitem}
\setlist[enumerate,1]{label={(\alph*)},itemsep=\parskip}

\newlist{thmlist}{enumerate}{1}
\setlist[thmlist,1]{
  label={\em(\roman*)}, ref={(\roman*)},
  itemsep=0.5em,
  topsep=0em,
  leftmargin=*,
  align=left,widest=vi)}

\newlist{thmlistbis}{enumerate}{1}
\setlist[thmlistbis,1]{
  label={\em(\roman*~\textit{bis})},
  ref={(\roman*}~\textit{bis}\upshape{)},
  itemsep=0.5em,
  topsep=-0.7em,
  leftmargin=0pt, align=right, widest=vi)}

\newlist{defnlist}{enumerate}{2}
\setlist[defnlist,1]{
  label={(\roman*)}, ref={(\roman*)},
  itemsep=0.5em,
  topsep=0em,
  leftmargin=*,
  align=left, widest=vi)}

\setlist[defnlist,2]{
  label={(\alph*)}, ref={(\alph*)},
  itemsep=0.75em,
  labelsep=0em,labelindent=0em,leftmargin=*,align=left,widest=vi),
  topsep=0.75em}

\newlist{inlinelist}{enumerate*}{1}
\setlist[inlinelist,1]{label={(\alph*)}}

\newlist{inlinedefnlist}{enumerate*}{1}
\definecolor{green}{HTML}{38550C}
\setlist[inlinedefnlist,1]{label={\color{green}(\roman*)}}

\newlist{inlinethmlist}{enumerate*}{1}
\definecolor{green}{HTML}{38550C}
\setlist[inlinethmlist,1]{label={\color{green}(\roman*)}}

\nc{\cA}{\ensuremath{\mathcal{A}}\xspace}
\nc{\cB}{\ensuremath{\mathcal{B}}\xspace}
\nc{\cC}{\ensuremath{\mathcal{C}}\xspace}
\nc{\cD}{\ensuremath{\mathcal{D}}\xspace}
\nc{\cE}{\ensuremath{\mathcal{E}}\xspace}
\nc{\cF}{\ensuremath{\mathcal{F}}\xspace}
\nc{\cG}{\ensuremath{\mathcal{G}}\xspace}
\nc{\cH}{\ensuremath{\mathcal{H}}\xspace}
\nc{\cI}{\ensuremath{\mathcal{I}}\xspace}
\nc{\cJ}{\ensuremath{\mathcal{J}}\xspace}
\nc{\cK}{\ensuremath{\mathcal{K}}\xspace}
\nc{\cL}{\ensuremath{\mathcal{L}}\xspace}
\nc{\cM}{\ensuremath{\mathcal{M}}\xspace}
\nc{\cN}{\ensuremath{\mathcal{N}}\xspace}
\nc{\cO}{\ensuremath{\mathcal{O}}\xspace}
\nc{\cP}{\ensuremath{\mathcal{P}}\xspace}
\nc{\cQ}{\ensuremath{\mathcal{Q}}\xspace}
\nc{\cR}{\ensuremath{\mathcal{R}}\xspace}
\nc{\cS}{\ensuremath{\mathcal{S}}\xspace}
\nc{\cT}{\ensuremath{\mathcal{T}}\xspace}
\nc{\cU}{\ensuremath{\mathcal{U}}\xspace}
\nc{\cV}{\ensuremath{\mathcal{V}}\xspace}
\nc{\cW}{\ensuremath{\mathcal{W}}\xspace}
\nc{\cX}{\ensuremath{\mathcal{X}}\xspace}
\nc{\cY}{\ensuremath{\mathcal{Y}}\xspace}
\nc{\cZ}{\ensuremath{\mathcal{Z}}\xspace}

\nc{\sA}{\ensuremath{\mathcal{A}}\xspace}
\nc{\sB}{\ensuremath{\mathcal{B}}\xspace}
\nc{\sC}{\ensuremath{\mathcal{C}}\xspace}
\nc{\sD}{\ensuremath{\mathcal{D}}\xspace}
\nc{\sE}{\ensuremath{\mathcal{E}}\xspace}
\nc{\sF}{\ensuremath{\mathcal{F}}\xspace}
\nc{\sG}{\ensuremath{\mathcal{G}}\xspace}
\nc{\sH}{\ensuremath{\mathcal{H}}\xspace}
\nc{\sI}{\ensuremath{\mathcal{I}}\xspace}
\nc{\sJ}{\ensuremath{\mathcal{J}}\xspace}
\nc{\sK}{\ensuremath{\mathcal{K}}\xspace}
\nc{\sL}{\ensuremath{\mathcal{L}}\xspace}
\nc{\sM}{\ensuremath{\mathcal{M}}\xspace}
\nc{\sN}{\ensuremath{\mathcal{N}}\xspace}
\nc{\sO}{\ensuremath{\mathcal{O}}\xspace}
\nc{\sP}{\ensuremath{\mathcal{P}}\xspace}
\nc{\sQ}{\ensuremath{\mathcal{Q}}\xspace}
\nc{\sR}{\ensuremath{\mathcal{R}}\xspace}
\nc{\sS}{\ensuremath{\mathcal{S}}\xspace}
\nc{\sT}{\ensuremath{\mathcal{T}}\xspace}
\nc{\sU}{\ensuremath{\mathcal{U}}\xspace}
\nc{\sV}{\ensuremath{\mathcal{V}}\xspace}
\nc{\sW}{\ensuremath{\mathcal{W}}\xspace}
\nc{\sX}{\ensuremath{\mathcal{X}}\xspace}
\nc{\sY}{\ensuremath{\mathcal{Y}}\xspace}
\nc{\sZ}{\ensuremath{\mathcal{Z}}\xspace}

\nc{\bA}{\ensuremath{\mathbf{A}}\xspace}
\nc{\bB}{\ensuremath{\mathbf{B}}\xspace}
\nc{\bC}{\ensuremath{\mathbf{C}}\xspace}
\nc{\bD}{\ensuremath{\mathbf{D}}\xspace}
\nc{\bE}{\ensuremath{\mathbf{E}}\xspace}
\nc{\bF}{\ensuremath{\mathbf{F}}\xspace}
\nc{\bG}{\ensuremath{\mathbf{G}}\xspace}
\nc{\bH}{\ensuremath{\mathbf{H}}\xspace}
\nc{\bI}{\ensuremath{\mathbf{I}}\xspace}
\nc{\bJ}{\ensuremath{\mathbf{J}}\xspace}
\nc{\bK}{\ensuremath{\mathbf{K}}\xspace}
\nc{\bL}{\ensuremath{\mathbf{L}}\xspace}
\nc{\bM}{\ensuremath{\mathbf{M}}\xspace}
\nc{\bN}{\ensuremath{\mathbf{N}}\xspace}
\nc{\bO}{\ensuremath{\mathbf{O}}\xspace}
\nc{\bP}{\ensuremath{\mathbf{P}}\xspace}
\nc{\bQ}{\ensuremath{\mathbf{Q}}\xspace}
\nc{\bR}{\ensuremath{\mathbf{R}}\xspace}
\nc{\bS}{\ensuremath{\mathbf{S}}\xspace}
\nc{\bT}{\ensuremath{\mathbf{T}}\xspace}
\nc{\bU}{\ensuremath{\mathbf{U}}\xspace}
\nc{\bV}{\ensuremath{\mathbf{V}}\xspace}
\nc{\bW}{\ensuremath{\mathbf{W}}\xspace}
\nc{\bX}{\ensuremath{\mathbf{X}}\xspace}
\nc{\bY}{\ensuremath{\mathbf{Y}}\xspace}
\nc{\bZ}{\ensuremath{\mathbf{Z}}\xspace}

\nc{\bbA}{\ensuremath{\mathbb{A}}\xspace}
\nc{\bbB}{\ensuremath{\mathbb{B}}\xspace}
\nc{\bbC}{\ensuremath{\mathbb{C}}\xspace}
\nc{\bbD}{\ensuremath{\mathbb{D}}\xspace}
\nc{\bbE}{\ensuremath{\mathbb{E}}\xspace}
\nc{\bbF}{\ensuremath{\mathbb{F}}\xspace}
\nc{\bbG}{\ensuremath{\mathbb{G}}\xspace}
\nc{\bbH}{\ensuremath{\mathbb{H}}\xspace}
\nc{\bbI}{\ensuremath{\mathbb{I}}\xspace}
\nc{\bbJ}{\ensuremath{\mathbb{J}}\xspace}
\nc{\bbK}{\ensuremath{\mathbb{K}}\xspace}
\nc{\bbL}{\ensuremath{\mathbb{L}}\xspace}
\nc{\bbM}{\ensuremath{\mathbb{M}}\xspace}
\nc{\bbN}{\ensuremath{\mathbb{N}}\xspace}
\nc{\bbO}{\ensuremath{\mathbb{O}}\xspace}
\nc{\bbP}{\ensuremath{\mathbb{P}}\xspace}
\nc{\bbQ}{\ensuremath{\mathbb{Q}}\xspace}
\nc{\bbR}{\ensuremath{\mathbb{R}}\xspace}
\nc{\bbS}{\ensuremath{\mathbb{S}}\xspace}
\nc{\bbT}{\ensuremath{\mathbb{T}}\xspace}
\nc{\bbU}{\ensuremath{\mathbb{U}}\xspace}
\nc{\bbV}{\ensuremath{\mathbb{V}}\xspace}
\nc{\bbW}{\ensuremath{\mathbb{W}}\xspace}
\nc{\bbX}{\ensuremath{\mathbb{X}}\xspace}
\nc{\bbY}{\ensuremath{\mathbb{Y}}\xspace}
\nc{\bbZ}{\ensuremath{\mathbb{Z}}\xspace}

\nc{\mrm}[1]{\ensuremath{\mathrm{#1}}\xspace}
\nc{\mit}[1]{\ensuremath{\mathit{#1}}\xspace}
\nc{\mbf}[1]{\ensuremath{\mathbf{#1}}\xspace}
\nc{\mcal}[1]{\ensuremath{\mathcal{#1}}\xspace}
\nc{\msc}[1]{\ensuremath{\mathscr{#1}}\xspace}

\nc{\sub}{\subseteq}
\nc{\too}{\longrightarrow}
\nc{\hook}{\hookrightarrow}
\nc{\hooklongrightarrow}{\lhook\joinrel\longrightarrow}
\nc{\hooklong}{\hooklongrightarrow}
\nc{\hooklongleftarrow}{\longleftarrow\joinrel\rhook}
\nc{\twoheadlongrightarrow}{\relbar\joinrel\twoheadrightarrow}
\nc{\longrightleftarrows}{\ \raisebox{0.3ex}{\(\mathrel{\substack{\xrightarrow{\rule{1em}{0em}} \\[-1ex] \xleftarrow{\rule{1em}{0em}}}}\)}\ }

\renc{\ge}{\geqslant}
\renc{\le}{\leqslant}

\nc{\id}{\mathrm{id}}

\DeclareMathOperator{\Hom}{\on{Hom}}
\nc{\uHom}{\underline{\smash{\Hom}}}

\DeclareMathOperator{\End}{\on{End}}

\nc{\uEnd}{\underline{\smash{\End}}}

\nc{\colim}{\varinjlim}
\renc{\lim}{\varprojlim}
\nc{\Cofib}{\on{Cofib}}
\nc{\Fib}{\on{Fib}}
\nc{\initial}{\varnothing}
\nc{\op}{\mathrm{op}}

\DeclareMathOperator*{\fibprod}{\times}

\renc{\setminus}{\smallsetminus}

\usepackage{mathtools}
\DeclarePairedDelimiter\abs{\lvert}{\rvert}%

\newcommand{\thmref}[1]{Theorem~\ref{#1}}

\newcommand{\secref}[1]{Sect.~\ref{#1}}

\newcommand{\sssecref}[1]{(\ref{#1})}
\newcommand{\lemref}[1]{Lemma~\ref{#1}}
\newcommand{\propref}[1]{Proposition~\ref{#1}}

\newcommand{\remref}[1]{Remark~\ref{#1}}

\renewcommand{\eqref}[1]{(\ref{#1})}

\newcommand{\itemref}[1]{\ref{#1}}

\nc{\A}{\bA}
\nc{\bDelta}{\mathbf{\Delta}}
\nc{\Cech}{\textnormal{\v{C}}}
\nc{\vir}{\mrm{vir}}
\renc{\H}{\on{H}}
\nc{\BM}{\mrm{BM}}
\nc{\Z}{\bZ}
\nc{\Q}{\bQ}
\nc{\rightrightrightarrows}
  {\mathrel{\substack{\textstyle\rightarrow\\[-0.6ex]
   \textstyle\rightarrow \\[-0.6ex]
   \textstyle\rightarrow}}}
\nc{\rightrightrightrightarrows}
  {\mathrel{\substack{\textstyle\rightarrow\\[-0.6ex]
   \textstyle\rightarrow \\[-0.6ex]
   \textstyle\rightarrow \\[-0.6ex]
   \textstyle\rightarrow}}}
\nc{\Einfty}{{\sE_\infty}}
\nc{\pr}{\mrm{pr}}
\nc{\pt}{\mrm{pt}}
\nc{\vb}[1]{\langle{#1}\rangle}

\nc{\C}{\on{C}}
\nc{\Chom}{\mrm{C}_\bullet}
\nc{\Ccoh}{\mrm{C}^\bullet}
\nc{\Ccohc}{\mrm{C}_{\mrm{c}}^\bullet}
\nc{\CBM}{\mrm{C}^{\BM}_\bullet}

\nc{\CCOH}{\mathsf{C}^\bullet}
\nc{\CCOHC}{\mathsf{C}_{\mathsf{c}}^\bullet}
\nc{\CHOM}{\mathsf{C}_\bullet}
\nc{\CHOMBM}{\mathsf{C}^{\mathsf{BM}}_\bullet}

\nc{\Spc}{\mrm{Spc}}
\nc{\Stk}{\mrm{Stk}}
\nc{\Art}{\mrm{Art}}
\nc{\Shv}{\on{Shv}}
\nc{\Spt}{\mrm{Spt}}
\nc{\Anima}{\mathbf{H}}
\nc{\Tot}{\on{Tot}}
\nc{\Ex}{\mrm{Ex}}
\nc{\unit}{\mrm{unit}}
\nc{\counit}{\mrm{counit}}
\nc{\gys}{\mrm{gys}}

\nc{\scr}{\term{derived commutative ring}}
\nc{\scrs}{\term{derived commutative rings}}

\nc{\inftyCat}{\term{$\infty$-category}}
\nc{\inftyCats}{\term{$\infty$-categories}}

\nc{\inftyGrpd}{\term{$\infty$-groupoid}}
\nc{\inftyGrpds}{\term{$\infty$-groupoids}}

\nc{\dA}{\term{derived Artin}}

\title{Equivariant homology of stacks\vspace{-2mm}}
\author[A.\,A. Khan]{Adeel A. Khan}
\date{2025-06-04}

\makeatletter
\def\l@subsection{\@tocline{2}{0pt}{4pc}{6pc}{}}
\makeatother

\begin{document}

\begin{abstract}
  We use sheaf theory and the six operations to define and study the (equivariant) homology of stacks.
  The construction makes sense in the algebraic, complex analytic, or even topological categories.
  \vspace{-5mm}
\end{abstract}

\maketitle

\setlength{\parindent}{1em}
\parskip 0.6em

\thispagestyle{empty}


\changelocaltocdepth{1}

\section*{Introduction}

Let $X$ be a locally compact and Hausdorff topological space.
The singular homology groups $\H_*(X; \bZ)$ can be described via sheaf theory.
Namely, one has the functors $f_!$ and $f^!$ of compactly supported direct and inverse image along the projection $f : X \to \pt$ on derived categories of sheaves of abelian groups.
By Grothendieck--Verdier duality, one has
\[
  \H_*(X; \bZ) \simeq \H^{-*}(f_!f^! \bZ).
\]
Similarly, Borel--Moore homology\footnote{%
  a.k.a. locally finite homology
} is computed by the complex $f_*f^! \bZ$.
Moreover, the various operations and properties of both variants fall out of the yoga of the six functors right away.
We refer to \cite{Godement,KashiwaraSchapira} for background on sheaves and the six functor formalism\footnote{%
  More modern treatments such as \cite{Volpe} allow us to treat unbounded derived categories, which is why we do not need to assume finiteness of cohomological dimension above.
} and \cite[Chap.~V]{Bredon} for the sheaf-theoretic point of view on homology and Borel--Moore homology.

In this note we are interested in generalizing the above story to define \emph{$G$-equivariant} homology, when $G$ is a topological group acting on $X$.
At the same time, we also want to allow $X$ to be a topological Artin\footnote{%
  or even \emph{higher} Artin
} stack.
The language of \inftyCats makes it easy to extend derived categories to stacks, and the six operations are also available in that setting by \cite{weavelisse} (see also \cite{LaszloOlsson,LiuZheng,KapranovVasserot}, and our review in \secref{sec:sheaves}).
The $G$-equivariant homology of $X$ is the \emph{relative} homology of the induced map of quotient stacks $g : [X/G] \to BG := [\pt/G]$, and similarly for the Borel--Moore variant.
That is to say,
\[
  \H_*^G(X; \Z) := \H^{-*}(BG, g_!g^! \bZ),
  \quad \H_*^{\BM,G}(X; \Z) := \H^{-*}(BG, g_*g^! \bZ).
\]
As above, all desired properties of these constructions are derived immediately from the six functor formalism.
Of course, they also specialize to homology theories in the complex analytic and algebraic categories.\footnote{%
  For readers uncomfortable with topological stacks, we note that it is not necessary to pass through underlying topological stacks; see \sssecref{sssec:algstkshv}.
}

One motivation to write this note came from work of D.~Joyce (see \cite{Joyce}), who formulates certain wall-crossing formulas in enumerative geometry in terms of vertex algebra structures defined on the homology of certain (algebraic) moduli stacks.
This in particular requires a suitable equivariant homology theory as above.
See also forthcoming work of A.~Bojko \cite{Bojko}, who uses our definition to refine Joyce's vertex algebras.\footnote{%
  Combining the discussion in Subsects.~\ref{ssec:homotype}, \ref{ssec:relhom} (particularly \remref{rem:lim1}), and \ref{ssec:Borel}, one can see that the definition proposed by Joyce is a quotient of ours.
}

Let us also note that the considerations of this paper make sense in any suitable six functor formalism\footnote{%
  to be precise: in any topological weave in the sense of \cite{weaves,weavelisse}
}.
For example, one could use the machinery of motivic sheaves (see \cite[App.~A]{virtual} or \cite[\S 7]{weavelisse}) to define equivariant motivic homology together with a cycle class map to equivariant singular homology.
Motivic homology can be thought of as ``compactly supported'' version of (higher) Chow groups ($\approx$ motivic Borel--Moore homology) and, unlike the latter (see \cite[Rem.~2.3(b)]{Joyce}), can be used to define a Chow type lift of Joyce's vertex algebra structure.

\ssec*{Acknowledgments}

I am grateful to Arkadij Bojko and Dominic Joyce for that discussions that led to this note as well as feedback on previous drafts.
I also thank Charanya Ravi for useful discussions, and James E. Hotchkiss for pointing out an error.

\changelocaltocdepth{2}

\section{Sheaves}
\label{sec:sheaves}

\subsection{Topological stacks}

\subsubsection{}

Let $\Spc$ denote the category of topological spaces, always implicitly assumed locally compact and Hausdorff.
We say that a morphism of topological spaces is \emph{smooth}\footnote{%
  We will never use this term to refer to differentiable structures.
} if it is a topological submersion.

\subsubsection{}

Let $\Stk$ denote the \inftyCat of topological stacks, i.e., sheaves of \inftyGrpds on $\Spc$.
We say that a topological stack $X$ is \emph{$0$-Artin} if it is a topological space.
We define $1$-Artinness as follows.
\begin{defnlist}
  \item A morphism $f : X \to Y$ is \emph{$0$-Artin} (or simply \emph{representable}) if for every $V\in\Spc$ and every morphism $V \to Y$, the fibred product $X \fibprod_Y V$ is $0$-Artin.
  \item A $0$-Artin morphism $f : X \to Y$ is \emph{smooth} if for every $Y\in\Spc$ and every morphism $V \to Y$, the base change $X\fibprod_Y V \to V$ is smooth.
  \item A topological stack $X$ is \emph{$1$-Artin} if its diagonal $\Delta_X : X \to X \times X$ is $0$-Artin, and there exists $U\in\Spc$ and a morphism\footnote{%
    which is automatically $0$-Artin when $\Delta_X$ is $0$-Artin
  } $p : U \twoheadrightarrow X$ which is a smooth surjection.
  \item If $Y$ is a space and $X$ is $1$-Artin, a morphism $f : X \to Y$ is \emph{smooth} if there exists a space $U$ and a smooth surjection $p : U \twoheadrightarrow X$ such that $f\circ p : U \to Y$ is a smooth morphism of spaces.
\end{defnlist}
Iterating this inductively\footnote{%
  noting at each point that if $X$ has $(n-1)$-representable diagonal, then any morphism $U \to X$ from a space $U$ is automatically $(n-1)$-representable
}, we define $k$-Artin stacks and (smooth) $k$-Artin morphisms for all $k \ge 0$.
We say $X$ is \emph{Artin} if it is $k$-Artin for some $k$, and a morphism is Artin if it is representable by Artin stacks.

\subsubsection{}

Since properness for morphisms of topological spaces is local on the target, we can extend this notion to topological Artin stacks as follows.

A $0$-Artin morphism $f : X \to Y$ is \emph{proper} if for any morphism $V \to Y$ with $V \in \Spc$, the base change $X \fibprod_Y V \to V$ is a proper morphism of spaces.

A $1$-Artin morphism $f : X \to Y$ is \emph{proper} if for any morphism $V \to Y$ with $V \in \Spc$, there exists an (automatically $0$-Artin) proper surjection $Z \twoheadrightarrow X \fibprod_Y V$ with $Z \in \Spc$ such that $Z \to V$ is proper.

\subsubsection{}\label{sssec:algstk}

The analogous definitions make sense also in the algebraic and complex analytic categories.

The complex analytification functor sends an algebraic space, locally of finite type over the field $k$ of complex numbers, to its space of $k$-points.
This extends uniquely to a colimit-preserving functor sending locally of finite type $k$-stacks in the algebraic category to stacks in the complex analytic or topological category.

This functor preserves Artin stacks, Artin morphisms, smooth morphisms, vector bundles, and proper representable morphisms.

While our definition of proper morphisms of $1$-Artin stacks is not the standard one in the algebraic category (cf. \cite[Tag~0CL4]{Stacks}), it follows from \cite[Thm.~1.1]{Olsson} and \cite[Tags~0CL3, 06TZ, 06U7, 0DTL]{Stacks} that they agree; hence proper morphisms of $1$-Artin stacks are also preserved by complex analytification.

When the discussion does not depend on any particular context, we will use the terms \emph{space} and \emph{stack} agnostically.
For example, when read in the algebraic category, \emph{space} will mean ``algebraic space locally of finite type over $k$''.

\subsection{Sheaves}

\subsubsection{}

Fix a commutative ring of coefficients $\Lambda$.
For a topological space $X$, we denote by $\Shv(X)$ the derived \inftyCat of sheaves of $\Lambda$-modules on $X$.\footnote{%
  Equivalently, $\Shv(X)$ is the \inftyCat of sheaves on $X$ with values in the derived \inftyCat of $\Lambda$-modules.
  More generally, one can take $\Lambda$ to be an $\Einfty$-ring spectrum, and $\Shv(X)$ to be the \inftyCat of sheaves of $\Lambda$-module spectra on $X$.
}
We extend $\Shv(-)$ to topological stacks by lisse extension, as in \cite[\S 8]{weavelisse}.
Thus for a topological stack $X$, we have
\begin{equation}\label{eq:ShvlimPt}
  \Shv(X) \simeq \lim_{(S,s)} \Shv(S)
\end{equation}
where the limit of \inftyCats is taken over pairs $(S,s)$ with $S\in\Spc$ and $s : S \to X$ a morphism.
If $X$ is Artin, it is enough to take pairs where $s$ is \emph{smooth}.

\subsubsection{}\label{sssec:algstkshv}

We may also perform the lisse extension within the algebraic or complex analytic categories.
For example, if $X$ is a stack in the algebraic category, we have
\[ \Shv(X) := \lim_{(S,s)} \Shv(S) \]
where the limit is taken over pairs $(S,s)$ with $S$ a locally of finite type algebraic space (or scheme) and $s : S \to X$ a morphism in the algebraic category.
For $X$ Artin, we may again require $s$ to be smooth.
See \cite[\S 3.2]{weavelisse}.

There is a natural equivalence $\Shv(X) \simeq \Shv(X^\mathrm{top})$, where $X^\mathrm{top}$ is the underlying topological stack \sssecref{sssec:algstk}.
This follows from the fact that $X \mapsto X^{\mathrm{top}}$ preserves colimits, and that $X \mapsto \Shv(X)$ sends colimits of stacks to limits of \inftyCats (see \cite[\S 3.2]{weavelisse}).

\subsubsection{}
The six functor formalism is summed up as follows.
By the discussion in \sssecref{sssec:algstk} and \sssecref{sssec:algstkshv}, it may be read in any of the topological, complex analytic, or algebraic categories according to the reader's preference.

\begin{thm}\label{thm:sixops}\leavevmode
  The assignment $X \mapsto \Shv(X)$ defines a topological weave in the sense of \cite{weaves,weavelisse}.
  In particular, we have:
  \begin{thmlist}
    \item\label{item:sixops/otimes}
    \emph{Tensor and Hom.}
    For every stack $X$, there is a closed symmetric monoidal structure on $\Shv(X)$.
    In particular, there are adjoint bifunctors $(\otimes, \uHom)$.
    The monoidal unit is the constant sheaf $\Lambda_X \in \Shv(X)$.

    \item\label{item:sixops/f^*}
    \emph{*-Functoriality.}
    For every morphism $f : X \to Y$, there is a pair of adjoint functors
      $$ f^* : \Shv(Y) \to \Shv(X),
      \quad f_* : \Shv(X) \to \Shv(Y). $$
    The functor $f^*$ (resp. $f_*$) is symmetric monoidal (resp. lax symmetric monoidal).
    The assignments $f \mapsto f^*$, $f \mapsto f_*$ are compatible with composition up to coherent homotopy.

    \item\label{item:sixops/f^!}
    \emph{!-Functoriality.}
    For every Artin morphism $f : X \to Y$, there is a pair of adjoint functors
      $$ f_! : \Shv(X) \to \Shv(Y),
      \quad f^! : \Shv(Y) \to \Shv(X). $$
    The assignments $f \mapsto f_!$, $f \mapsto f^!$ are compatible with composition up to coherent homotopy.
    
    \item\label{item:sixops/basechange}
    \emph{Base change.}
    For every Artin morphism $f : X \to Y$, the operation $f_!$ commutes with $*$-inverse image.
    Dually, $f^!$ commutes with $*$-direct image.

    \item\label{item:sixops/projection}
    \emph{Projection formula.}
    Let $f : X \to Y$ be an Artin morphism.
    For every object $\sF \in \Shv(X)$ and every object $\sG \in \Shv(Y)$, there are canonical isomorphisms
    \begin{align*}
      f_!(\sF) \otimes \sG &\to f_!(\sF \otimes f^*(\sG)),\\
      \uHom(f_!(\sF), \sG) &\to f_*\uHom(\sF, f^!(\sG)),\\
      f^!\uHom(\sF, \sG) &\to \uHom(f^*(\sF), f^!(\sG)).
    \end{align*}

    \item\label{item:sixops/alpha}
    \emph{Forgetting supports.}
    For every $1$-Artin morphism $f : X \to Y$ with proper diagonal, there is a natural transformation $\alpha_f : f_! \to f_*$ which is compatible with composition up to coherent homotopy.
    Moreover, it is invertible if $f$ is proper and representable.\footnote{%
      See also \sssecref{sssec:proper DM}.
    }

    \item\label{item:sixops/localization}
    \emph{Localization.}
    Let $X$ be an Artin stack and let $i : Z \to X$ and $j : U \to X$ be a complementary pair of closed and open immersions, respectively.\footnote{\label{fn:complement}%
      The open complement $U = X \setminus Z$ is by definition the subfunctor of $X$ whose $T$-points are $T \to X$ for which $T \fibprod_X Z = \initial$ (for every $T\in\Spc$).
      That is to say, it is the largest substack of $X$ that doesn't intersect $U$.
      It is easy to see that the morphism $U \to X$ is an open immersion.
    }
    Then there are exact triangles of functors
    \begin{align*}
      j_!j^* \to \id \to i_*i^*,\\
      i_*i^! \to \id \to j_*j^!.
    \end{align*}

    \item\label{item:sixops/homotopy}
    \emph{Homotopy invariance.}
    Let $X$ be a stack and $\pi : E \to X$ be the projection of a vector bundle.
    Then the natural transformations
    \begin{align*}
      \id &\to \pi_*\pi^*,\\
      \pi_!\pi^! &\to \id,
    \end{align*}
    are invertible.

    \item\label{item:sixops/Thom}
    \emph{Thom twist.}
    Let $f : X \to Y$ be a smooth Artin morphism.
    Denote by $\Delta : X \to X \fibprod_Y X$ and $\pr_2 : X \fibprod_Y X \to X$ the diagonal and second projection, respectively.
    Then the assignment $$\sF \mapsto \sF\vb{-T_f} := \Delta^! \pr_2^* (\sF)$$ determines a $\Shv(X)$-linear auto-equivalence of $\Shv(X)$, whose inverse we denote $\sF \mapsto \sF\vb{T_f}$.\footnote{\label{fn:Tf}%
      In the topological category, ``$T_f$'' is here just a symbol.
      The notation is meant to suggest that but one should think of $\Lambda_X\vb{T_f}$ as the \emph{Thom sheaf of the relative tangent microbundle} of $f$.
      In the algebraic or complex analytic (or even $C^1$) categories, $\Lambda_X\vb{T_f}$ can indeed be identified with the Thom sheaf of the relative tangent bundle $T_f$, via a deformation to the normal bundle argument.
    }

    \item\label{item:sixops/purity}
    \emph{Poincaré duality.}
    Let $f : X \to Y$ be an Artin morphism.
    If $f$ is smooth, then there is a canonical isomorphism of functors
    \begin{equation*}
      f^!(-) \simeq f^*(-)\vb{T_f}.
    \end{equation*}
    In particular, $\omega_{X/Y} := f^!(\Lambda_Y) \simeq \Lambda_X\vb{T_f}$ is $\otimes$-invertible.
  \end{thmlist}
\end{thm}

See \sssecref{sssec:proofsixops} below for the proof.

\subsubsection{}

Let $f : X \to Y$ be a smooth morphism of relative dimension $d$ between Artin stacks.
An \emph{orientation} of $f$ is a trivialization $\Lambda_X\vb{T_f} \simeq \Lambda_X[d]$.
(Note that this depends on the coefficient ring $\Lambda$.)
Given an orientation, Poincaré duality reads $f^! \simeq f^*\vb{d}$.

For example, any smooth morphism of Artin stacks in the algebraic or complex analytic categories admits a canonical orientation.\footnote{%
  An orientation of $f$ amounts to a Thom isomorphism for $T_f$ (in view of Footnote~\ref{fn:Tf}).
  If $\Lambda$ is an $\Einfty$-ring spectrum equipped with a complex orientation, such as the Eilenberg--MacLane spectrum for an ordinary commutative ring, the complex orientation determines compatible Thom isomorphisms for all complex vector bundles.
}

\subsubsection{}
\label{sssec:more base change}

For any cartesian square of stacks
\[\begin{tikzcd}
  X' \ar{r}{g}\ar{d}{p}
  & Y' \ar{d}{q}
  \\
  X \ar{r}{f}
  & Y
\end{tikzcd}\]
where $q$ is an Artin morphism, we have in addition to the base change isomorphisms
\[ \Ex^*_! : q^*f_! \simeq g_!p^*, \quad \Ex^!_* : q^!f_* \simeq g_*q^!, \]
further exchange transformations
\[
  \Ex^*_* : q^*f_* \to g_*p^*,
  \quad \Ex^{*,!} : g^*q^! \to p^!f^*,
\]
see \cite[\S A.2]{weavelisse}, which also become invertible in certain situations:

\begin{enumerate}
  \item 
  Whenever we have the forgetting supports isomorphism $f_! \simeq f_*$ (and similarly for $g$), one deduces that $\Ex^*_*$ is invertible.
  (For example, when $f$ and $g$ are proper representable morphisms of $1$-Artin stacks.)
  
  \item
  When $p$ and $q$ are smooth and hence satisfy Poincaré duality, one deduces that $\Ex^*_*$ and $\Ex^{*,!}$ are invertible.
\end{enumerate}

\subsubsection{}

We also have the following descent statements for the weave $\Shv(-)$:

\begin{prop}[Descent]\label{prop:desc}
  Let $p : Y \to X$ be an Artin morphism and denote by $Y_\bullet$ its \v{C}ech nerve.
  \begin{thmlist}
    \item\label{item:desc/smdescent}
    \emph{Smooth descent.}
    Consider the canonical functor of \inftyCats
    \[
      \Shv(X) \to \on{Tot}(\Shv(Y_\bullet)),
    \]
    where $\on{Tot}(-)$ denotes the totalization (homotopy limit), with transition functors given by $*$-inverse image (resp. by $!$-inverse image).
    If $p$ is a smooth surjection, then this is an equivalence.

    \item\label{item:desc/propdescent}
    \emph{Proper descent.}
    Consider the canonical functor of \inftyCats
    \[
      \widehat{\Shv}(X) \to \on{Tot}(\widehat{\Shv}(Y_\bullet)),
    \]
    where the transition functors are given by $*$-inverse image and $\widehat{\Shv}(-)$ denotes the left completion\footnote{%
      So when $X$ is a topological space, $\widehat{\Shv}(X)$ is the completed derived \inftyCat of sheaves of abelian groups on $X$.
      See e.g. \cite[\S 1.2]{Haine} or \cite[\S C.5.9]{SAG}.
      The canonical functor $\Shv(X) \to \widehat{\Shv}(X)$ is an equivalence when $X$ is a CW complex, and is always an equivalence on eventually coconnective (homologically bounded above) objects.
    } of $\Shv(-)$ with respect to the cohomological t-structure\footnote{\label{fn:cohtstr}%
      extended to topological Artin stacks as in \cite[Prop.~4.3]{equilisse}, i.e. by imposing that smooth $*$-inverse image is t-exact.
    }.
    If $p$ is a proper representable surjection, then this is an equivalence.
  \end{thmlist}
\end{prop}

\subsubsection{Proof of \thmref{thm:sixops} and \propref{prop:desc}}
\label{sssec:proofsixops}

See \cite[\S 7.1]{weavelisse} and \cite[\S 8.3]{weavelisse} for the construction of the weave $\Shv(-)$ in the algebraic and topological categories, respectively (and \cite[Cor.~3.4.9]{weavelisse} for the abstract statement).
In particular, parts~\itemref{item:sixops/otimes}, \itemref{item:sixops/f^*}, \itemref{item:sixops/f^!}, \itemref{item:sixops/basechange}, \itemref{item:sixops/projection}, \itemref{item:sixops/Thom}, and \itemref{item:sixops/purity} of \thmref{thm:sixops} are proven there, as well as \propref{prop:desc}\itemref{item:desc/smdescent}.
Using \propref{prop:desc}\itemref{item:desc/smdescent} and \sssecref{sssec:more base change}, parts~\itemref{item:sixops/localization} and \itemref{item:sixops/homotopy} of \thmref{thm:sixops} then reduce to the case of spaces.
Similarly, \propref{prop:desc}\itemref{item:desc/propdescent} reduces\footnote{%
  Note that left completion commutes with lisse extension for Artin stacks, since $*$-inverse image along smooth morphisms are t-exact, as in \eqref{eq:ShvlimPt}.
} to the case of proper morphisms of topological spaces, proven in \cite[Cor.~2.8]{Haine}.

\thmref{thm:sixops}\itemref{item:sixops/alpha} is only proven for representable morphisms in \cite[\S 8.3]{weavelisse}.
Let us construct $\alpha_f : f_! \to f_*$ for $f : X \to Y$ a $1$-Artin morphism whose diagonal $\Delta$ is proper.
Since $\Delta$ is representable, we have the isomorphism $\alpha_{\Delta} : \Delta_! \simeq \Delta_*$.
By base change (\thmref{thm:sixops}\itemref{item:sixops/basechange}) for the cartesian square
\[\begin{tikzcd}
  X\fibprod_Y X \ar{r}{\pr_2}\ar{d}{\pr_1}
  & X \ar{d}{f}
  \\
  X \ar{r}{f}
  & Y
\end{tikzcd}\]
we get the natural transformation
\begin{equation}
  f^*f_!
  \simeq \pr_{2,!}\pr_1^*
  \xrightarrow{\unit} \pr_{2,!}\Delta_!\Delta^*\pr_1^*
  \simeq \id,
\end{equation}
whence $\alpha_f : f_! \to f_*$ by transposition.

Now suppose $f : X \to Y$ is a proper morphism of $1$-Artin stacks and let us show that $\alpha_f$ is invertible.
Choose a smooth surjection $v : V \twoheadrightarrow Y$ where $V\in\Spc$.
By \propref{prop:desc}\itemref{item:desc/smdescent}, it is enough to show that $v^*f_! \to v^*f_*$ is invertible.
Since $f_!$ (resp. $f_*$) commutes with arbitrary (resp. smooth) $*$-inverse image \sssecref{sssec:more base change}, we may replace $Y$ by $V$ and thereby reduce to the case where $Y\in\Spc$.
By definition of properness of $f$, there exists a proper surjection $g : Z \twoheadrightarrow X$ where $Z\in\Spc$ and $g\circ f : Z \to Y$ is proper.
By \propref{prop:desc}\itemref{item:desc/propdescent} we have descent along the \v{C}ech nerve $g_\bullet : Z_\bullet \to X$ of $g$, so it will suffice to show that
\[\alpha_f : f_!g_{n,*}g_n^* \to f_*g_{n,*}g_n^*\]
is invertible for each $n$.
This follows from the compatibility of $\alpha$ with composition, the isomorphisms $\alpha_{{g_n},*} : g_{n,!} \to g_{n,*}$, and the isomorphisms $\alpha_{{f \circ g_n},*} : (f\circ g_n)_! \to (f\circ g_n)_*$.

\subsubsection{}
\label{sssec:proper DM}

Let $Y$ be $1$-Artin and $f : X \to Y$ a proper Deligne--Mumford\footnote{%
  i.e., representable by Deligne--Mumford stacks.
  See \cite[\S 14]{NoohiFound} for the definition of Deligne--Mumford stacks in the topological category.
} morphism.
Over every point $y \in Y$, the fibre $X \fibprod_Y \{y\}$ has finite discrete stabilizers; we assume their orders are all invertible in the commutative ring $R$.
In this situation, one can show that the ``forget supports'' transformation $\alpha_f : f_! \to f_*$ is still invertible on cohomologically bounded objects $\sF \in \Shv(X; R)$.

The basic case is when $f$ is the projection of the classifying stack $B_Y(G)$ of a relative group space $G$ over $Y$ whose structural morphism $G \to Y$ is finite étale\footnote{%
  in the topological category, read: a proper local homeomorphism
}, in which case $\alpha_f$ is invertible on $\Shv(X; R)$ as soon as the order of $G$ is invertible in $R$.

We sketch a proof in the general case.
It is enough to work in the topological category.
We begin by showing that $f_!$ is left t-exact on the left completion $\widehat{\Shv}(X; R)$.
We may assume that $Y$ is a space.
Consider the underlying space $\abs{X}$ of $X$, i.e. the set $\pi_0(X(\pt))$ with opens $\abs{U}$ for every open substack $U \sub X$.
Then $f$ factors through the projection $\pi : X \to \abs{X}$ and the morphism of spaces $\abs{f} : \abs{X} \to \abs{Y} \simeq Y$.
One can check that, when $f$ is proper, $\pi$ and $\abs{f}$ are proper.
We know that $\abs{f}_! \simeq \abs{f}_*$ is left t-exact, so by functoriality it will suffice to consider $\pi_!$.
On $\widehat{\Shv}$ the functors $x^*$ are t-exact and jointly conservative as $x$ ranges among points of $\abs{X}$ (see e.g. \cite[Obs.~1.30]{Haine}), so by the base change formula it is enough to show that $\pi_{x,!}$ is left t-exact for each fibre $\pi_x : X_x \to \{x\}$.
By \cite[Prop.~10.2(v)]{NoohiFound}, this fibre is the projection $BG \to \{x\}$ where $G$ is the stabilizer group space of $X$ at $x$.
Since the order of $G$ is invertible in $R$ by assumption, $\pi_{x,!}\simeq \pi_{x,*}$ is indeed left t-exact by above.

Now let us show that $\alpha_f$ is invertible on a cohomologically bounded object $\sF \in \Shv(X; R)$.
It is enough to consider the case of $\sF \in \Shv(X; R)^\heartsuit$ discrete.
By smooth base change we may assume $Y$ is a space.
By definition of properness of $f$, there exists a proper surjection $g : Z \twoheadrightarrow X$ where $Z\in\Spc$ and $g\circ f : Z \to Y$ is proper.
By \propref{prop:desc}\itemref{item:desc/propdescent} we have descent along the \v{C}ech nerve $g_\bullet : Z_\bullet \to X$ of $g$, i.e., $\sF \simeq \Tot(g_{\bullet,*}g_\bullet^*\sF)$.
Since totalizations in the $1$-categories $\Shv(X; R)^\heartsuit$ and $\Shv(Y; R)^\heartsuit$ can be computed as finite limits (equalizers), and $f_! : \Shv(X; R)^\heartsuit \to \Shv(Y; R)^\heartsuit$ is exact, we find that $\alpha_f : f_!(\sF) \to f_*(\sF)$ is the limit over $[n]\in\bDelta$ of the maps
\[\alpha_f : f_!g_{n,*}g_n^*(\sF) \to f_*g_{n,*}g_n^*(\sF). \]
But these are all invertible, as follows from the compatibility of $\alpha$ with composition, the isomorphisms $\alpha_{{g_n},*} : g_{n,!} \to g_{n,*}$, and the isomorphisms $\alpha_{{f \circ g_n},*} : (f\circ g_n)_! \to (f\circ g_n)_*$.

\section{Homology}

\subsection{Definition}

As before, we fix a commutative ring\footnote{%
  and also as before, one can take more generally an $\Einfty$-ring spectrum
} of coefficients $\Lambda$, and work in either the topological, complex analytic, or algebraic category.
Let $f : X \to S$ be an Artin morphism of stacks and $\Lambda_S \in \Shv(S)$ the constant sheaf.

We define the following objects of $\Shv(S)$:
\begin{defnlist}
  \item\emph{Cochains:} $\CCOH(X_{/S}; \Lambda) := f_*f^*(\Lambda_S)$.
  \item\emph{Compactly supported cochains:} $\CCOHC(X_{/S}; \Lambda) := f_!f^*(\Lambda_S)$.
  \item\emph{Borel--Moore chains:} $\CHOMBM(X_{/S}; \Lambda) := f_*f^!(\Lambda_S)$.
  \item\emph{Chains:} $\CHOM(X_{/S}; \Lambda) := f_!f^!(\Lambda_S)$.
\end{defnlist}
We denote their respective complexes of (derived) global sections by
\[
  \Ccoh(X_{/S}; \Lambda), \quad
  \Ccohc(X_{/S}; \Lambda), \quad
  \CBM(X_{/S}; \Lambda), \quad
  \Chom(X_{/S}; \Lambda),
\]
respectively, and their respective cohomology groups by
\[
  \H^*(X_{/S}; \Lambda),
  \quad \H^*_{\mrm{c}}(X_{/S}; \Lambda),
  \quad \H_{-*}^\BM(X_{/S}; \Lambda),
  \quad \H_{-*}(X_{/S}; \Lambda).
\]
We leave the coefficient $\Lambda$ implicit when there is no risk of ambiguity.
When $S=\pt$, we omit the decoration $_{/S}$.
By adjunction, we may write
\begin{align*}
  \Ccoh(X_{/S}; \Lambda) &\simeq \Gamma(X, \Lambda_X) \simeq \Ccoh(X; \Lambda),\\
  \CBM(X_{/S}; \Lambda) &\simeq \Gamma(X, \omega_{X/S})
\end{align*}
where $\Lambda_X \simeq f^*(\Lambda_S)$ and $\omega_{X/S} := f^!(\Lambda_S)$.

\subsection{Operations}

The various compatibilities between the six operations give rise to the expected operations on (Borel--Moore) homology.
We record these for the sheaves $\CHOM(X_{/S})$ and $\CHOMBM(X_{/S})$, leaving the ring of coefficients $\Lambda$ implicit in the notation.

\subsubsection{Forgetting supports}

If $f: X \to S$ is a morphism of $1$-Artin stacks with proper diagonal, 
then $\alpha_f : f_! \to f_*$ induces a map
\[ \CHOM(X_{/S}) \to \CHOMBM(X_{/S}) \]
which is invertible if $X \to S$ is proper and Deligne--Mumford.

\subsubsection{Products}

Let $f : X \to Y$ and $g : Y \to S$ be Artin morphisms.
We have the operations of \emph{composition product}:
\begin{align}
  \circ &: g_*\CHOM(X_{/Y}) \otimes \CHOM(Y_{/S}) \to \CHOM(X_{/S}),\label{eq:caphom}\\
  \circ &: g_*\CHOMBM(X_{/Y}) \otimes \CHOMBM(Y_{/S}) \to \CHOMBM(X_{/S}).\label{eq:capBM}
\end{align}
Taking $f=g=\id$, these become the \emph{cup product} on cochains:
\begin{equation}
  \cup : \CCOH(Y) \otimes \CCOH(Y) \to \CCOH(Y),
\end{equation}
via which $\CCOH(Y)$ becomes an $\Einfty$-algebra in $\Shv(Y)$.
Taking $f=\id$, they become the \emph{cap products}:
\begin{align}
  \cap &: g_*\CCOH(Y) \otimes \CHOMBM(Y_{/S}) \to \CHOMBM(Y_{/S}),\\
  \cap &: g_*\CCOH(Y) \otimes \CHOM(Y_{/S}) \to \CHOM(Y_{/S}),
\end{align}
via which $\CHOMBM(Y_{/S})$ and $\CHOM(Y_{/S})$ are modules over $g_*\CCOH(Y)$ for any Artin morphism $g : Y \to S$.

We recall the construction of \eqref{eq:caphom}; that of \eqref{eq:capBM} is similar.
It comes from a natural transformation
\begin{equation}
  g_*f_!f^!g^*(-) \otimes g_!g^!(-)
  \to (g \circ f)_! (g \circ f)^!(-)
\end{equation}
defined as the composite
\begin{align*}
  g_*f_!f^!g^*(-) \otimes g_!g^!(-)
  \simeq &~g_!(g^*g_*f_!f^!g^*(-) \otimes g^!(-))\\
  \to &~g_!(g^!(-) \otimes f_!f^!g^*(-))\\
  \simeq &~g_!f_!(f^*g^!(-) \otimes f^!g^*(-))\\
  \to &~g_!f_!(f^!(g^*(-) \otimes g^!(-)))\\
  \to &~g_!f_!(f^!g^!(-))\\
  \simeq &~(g \circ f)_! (g \circ f)^!(-),
\end{align*}
where the (iso)morphisms are: the projection formula for $g_!$; the unit $\id \to g^*g_*$; the projection formula for $f_!$; the exchange transformation $\Ex^{*,!}_{\otimes} : f^*(-) \otimes f^!(-) \to f^!(-\otimes -)$; and the exchange transformation $\Ex^{*,1}_{\otimes} : g^*(-) \otimes g^!(-) \to g^!(-\otimes -)$.
(See \cite[A.1.2]{weavelisse} for the definition of the latter two.)

\subsubsection{Functoriality: change of base}

Let $p : S' \to S$ be a morphism of stacks.
For any Artin morphism $X \to S$, the unit $\id \to p_*p^*$ induces maps
\begin{align*}
  \CHOM(X_{/S}) &\to \CHOM({X_{S'}}_{/S'}),\\
  \CHOMBM(X_{/S}) &\to \CHOMBM({X_{S'}}_{/S'}),
\end{align*}
where $X_{S'} = X\fibprod_S S'$.

\subsubsection{Functoriality: direct image}

Let $f : X \to Y$ be a morphism of relatively Artin stacks over a stack $S$.
The counit $f_!f^! \to \id$ induces a map
\[ f_* : \CHOM(X_{/S}) \to \CHOM(Y_{/S}). \]
If $f$ is a \emph{proper} Deligne--Mumford morphism of $1$-Artin stacks, so that $f_*\simeq f_!$, then $f_*f^! \simeq f_!f^! \to \id$ induces a map
\[
  f_* : \CHOMBM(X_{/S}) \to \CHOMBM(Y_{/S}).
\]

\subsubsection{Functoriality: Gysin}
\label{sssec:Gys}

Let $f : X \to Y$ be a smooth Artin morphism of relative dimension $d$ over a stack $S$.
The Poincaré duality isomorphism $f^*\vb{T_f} \simeq f^!$ yields a natural transformation $\id \to f_*f^* \simeq f_*f^!\vb{-T_f}$, whence a map
\[ f^! : \CHOMBM(Y_{/S}) \to \CHOMBM(X_{/S})\vb{-T_f}, \]
where by abuse of notation we write, with $p : X \to S$ the structural morphism,
\[ \CHOMBM(X_{/S})\vb{-T_f} := p_*\vb{-T_f}p^!(\Lambda_S). \]
Recall that we may trivialize $\vb{-T_f} \simeq [-d]$ given a suitable orientation of $f$.

Combining this with the isomorphism $\alpha_f : f_! \to f_*$ when $f$ is moreover proper and Deligne--Mumford with $Y$ $1$-Artin, we have the natural transformation $\id \to f_*f^* \simeq f_!f^!\vb{-T_f}$ which yields a map
\[ f^! : \CHOM(Y_{/S}) \to \CHOM(X_{/S})\vb{-T_f}. \]

\subsubsection{Fundamental class}

Let $f : X \to S$ be a smooth Artin morphism.
The \emph{relative fundamental class}
\[ [X_{/S}] \in \CHOMBM(X_{/S})\vb{-T_f} \]
is the image of $1 \in \CCOH(S) \simeq \CHOMBM(S_{/S})$ by the Gysin map $f^! : \CHOMBM(S_{/S}) \to \CHOMBM(X_{/S})\vb{-T_f}$.\footnote{%
  Here $x \in \CHOMBM(X)\vb{-T_f}$ is shorthand for any of the following equivalent things: a morphism $x : \Lambda_X \to \CHOMBM(X)\vb{-T_f}$ in $\Shv(X)$; a morphism $x : \Lambda \to \CHOMBM(X)\vb{-T_f}$ in the derived \inftyCat of $\Lambda$-modules; an object of the \inftyGrpd underlying the complex $\CBM(X)\vb{-T_f}$.
}.

\subsubsection{Virtual functoriality}

Let us work in the $C^\infty$-category; we refer to \cite{Steffens} for background on (derived) $C^\infty$-stacks.
Given a derived $C^\infty$-stack $X$ and $\pi : E \to X$ a $C^\infty$-vector bundle, we consider the construction\footnote{%
  where $\Shv(-)$ and the various operations are defined via underlying topological stacks
}
\[ \sF\vb{-E} := 0^! \pi^*(\Lambda_X) \]
as an endofunctor in $\sF \in \Shv(X)$, where $0 : X \to E$ denotes the zero section.
This is an auto-equivalence (as can be checked locally) and we denote its inverse by $\sF \mapsto \sF\vb{E}$.
As in \cite[\S 2.6]{weaves}, compatibility with exact sequences of vector bundles implies that this extends to a canonical homomorphism
\[ \on{K_0}(\on{Perf}(X)) \to \Shv(X) \]
from the Grothendieck group of perfect complexes on $X$ to the Picard group of $\otimes$-invertible objects of $\Shv(X)$.

Now let $f : X \to Y$ be a \emph{quasi-smooth morphism of derived $C^\infty$-stacks} (see again \cite{Steffens} for a reference for quasi-smoothness in this context).
The relative cotangent complex is perfect and gives rise to a Thom twist $\vb{T_f^{\mathrm{vir}}}$ on $\Shv(X)$.
One can define in this situation a canonical natural transformation, the \emph{virtual Gysin transformation}
\[
  \gys_f : f^*(-)\vb{T_f^{\mathrm{vir}}} \to f^!(-),
\]
which reduces to the Poincaré duality isomorphism $f^! \simeq f^*(-)\vb{T_f}$ when $f$ is smooth.
See \cite[\S 3.1]{virtual} for the construction of $\gys_f$ in the algebraic category; the same works word-for-word in the complex analytic and $C^\infty$-categories, see e.g. \cite[\S 4]{PortaYuGW}.\footnote{%
  In fact, the construction does not require the full data of a derived structure on $f : X \to Y$.
  The latter determines in particular a closed immersion of the intrinsic normal cone $\mathfrak{C}_{X/Y}$ into the virtual normal bundle $N^{\mathrm{vir}}_{X/Y}$, and this linear shadow of the derived structure is enough: see \cite[Varnt.~2.4]{KhanCohdag} or \cite[\S 3.3]{virtual}.
}

Now, as in \sssecref{sssec:Gys}, $\gys_f$ gives rise to virtual Gysin maps:
\[ f^!_{\mathrm{vir}} : \CHOMBM(Y_{/S}) \to \CHOMBM(X_{/S})\vb{-T_f^{\mathrm{vir}}} \]
when $X$ and $Y$ are defined over some base $S$.
In particular, taking $Y=S$ we get the (relative) \emph{virtual fundamental class}
\[
  [X_{/S}]^{\mathrm{vir}} := f^!_{\mathrm{vir}}(1) \in \CHOMBM(X_{/S})\vb{-T_{X/S}^{\mathrm{vir}}}.
\]
If $f$ is moreover proper and Deligne--Mumford with $Y$ $1$-Artin, we also have
\[ f^!_{\mathrm{vir}} : \CHOM(Y_{/S}) \to \CHOM(X_{/S})\vb{-T_f^{\mathrm{vir}}}. \]

For example, the above discussion applies to any quasi-smooth Artin morphism of derived stacks in the algebraic or complex analytic categories (see e.g. \cite{stacksncts} or \cite{PortaYu} for some background on derived stacks in the respective contexts).
In that case, the canonical orientation of $f$ yields a canonical trivialization $\vb{T_f^{\mathrm{vir}}} \simeq [2d]$ where $d$ is the relative complex dimension.\footnote{%
  If $\Lambda$ is an $\Einfty$-ring spectrum, as opposed to an ordinary commutative ring, we need a complex orientation on $\Lambda$ here.
}

\subsection{Properties}

The following statements immediately follow from various parts of \thmref{thm:sixops}.

\begin{thm}\label{thm:props}
  Let $S$ be a stack and $f : X \to S$ an Artin morphism.
  \begin{thmlist}
    \item\label{item:loc}
    \emph{Localization.}
    For any closed immersion $i : Z \hook X$ with complementary open immersion $j : U \hook X$, there are exact triangles in $\Shv(S)$.
    \begin{align*}
      \CHOMBM(Z_{/S}) \xrightarrow{i_*} \CHOMBM(X_{/S}) \xrightarrow{j^!} \CHOMBM(U_{/S}),\\
      \CHOM(U_{/S}) \xrightarrow{j_*} \CHOM(X_{/S}) \to \CHOM((X,U)_{/S}),
    \end{align*}
    where $\Chom((X,U)_{/S}) = f_!i_!i^*f^!(\Lambda_S)$ is the complex of relative chains on $(X,U)$ over $S$.

    \item\label{item:htp}
    \emph{Homotopy invariance.}
    For any vector bundle $\pi : E \to X$ of rank $r$, the maps in $\Shv(S)$
    \begin{align*}
      &\pi^! : \CHOMBM(X_{/S}) \to \CHOMBM(E_{/S})\vb{-T_\pi},\\
      &\pi_* : \CHOM(E_{/S}) \to \CHOM(X_{/S})
    \end{align*}
    are invertible.

    \item\label{item:poinc}
    \emph{Poincaré duality.}
    If $f : X \to S$ is smooth, cap product with the fundamental class induces a canonical isomorphism in $\Shv(S)$.
    \[
      (-) \cap [X_{/S}] : \CCOH(X_{/S}) \to \CHOMBM(X_{/S})\vb{-T_{X/S}}.
    \]

    \item\label{item:ayog1}
    \emph{Smooth descent.}
    For any smooth surjective Artin morphism $p : Y \to X$ with \v{C}ech nerve $Y_\bullet$, there are canonical isomorphisms in $\Shv(S)$
    \begin{align*}
      \CHOMBM(X_{/S})
      &\simeq \lim_{[n]\in\bDelta} \CHOMBM({Y_n}_{/S})\vb{-T_{Y_n/X}},\\
      \CHOM(X_{/S})
      &\simeq \colim_{[n]\in\bDelta^\op} \CHOM({Y_n}_{/S}).
    \end{align*}
    
    \item\label{item:bphpjq}
    \emph{Proper descent.}
    Suppose that the relative dualizing complex $\omega_{X/S} = f^!(\Lambda_S)$ is eventually coconnective (e.g., $X \to S$ is a relative CW complex locally on the source and target).
    For any proper surjective morphism $p : Y \to X$ with \v{C}ech nerve $Y_\bullet$, there are canonical isomorphisms
    \begin{align*}
      \CHOMBM(X_{/S})
      &\simeq \colim_{[n]\in\bDelta^\op} \CHOMBM({Y_n}_{/S}),\\
      \CHOM(X_{/S})
      &\simeq \colim_{[n]\in\bDelta^\op} \CHOM({Y_n}_{/S})
    \end{align*}
    in $\Shv(S)$.
  \end{thmlist}
\end{thm}

\subsection{Homology and homotopy types}
\label{ssec:homotype}

We define the \emph{homotopy type} of a topological stack, cf. \cite{Simpson,Blanc,NoohiType}.
Denote by $\Anima$ the \inftyCat of homotopy types (a.k.a. \inftyGrpds).
We have the tautological functor
\[\Spc \to \Anima, \quad X \mapsto \abs{X},\]
which exhibits its target as the $\infty$-categorical localization at weak homotopy equivalences.
By left Kan extension, this extends uniquely to a colimit-preserving functor
\[\Spc \to \Anima, \quad X \mapsto \abs{X}.\]
In particular, $\abs{X}$ is canonically presented as the homotopy colimit
\[ \abs{X} \simeq \colim_{(S,s)} \abs{S} \]
taken over pairs $(S,s)$ with $S\in\Spc$ and $s : S \to X$ a morphism.

For a topological stack $X$, the complex of chains $\Chom(X)$ can be described in terms of its homotopy type $\abs{X} \in \Anima$.
Recall first of all that on topological spaces, the functor $X \mapsto \Chom(X)$ descends to a functor
\[\Chom^{\Anima}(-) : \Anima \to D(\Lambda)\]
valued in the derived \inftyCat of $\Lambda$-modules (where $\Lambda$ is the ring of coefficients), since it inverts weak homotopy equivalences.

The following is a tautological reformulation of smooth codescent for $\Chom(-) : \Stk \to D(\Lambda)$ (\thmref{thm:props}\itemref{item:ayog1}).

\begin{prop}
  For every topological stack $X$, there is a canonical natural isomorphism
  \[ \Chom(X) \simeq \Chom^{\Anima}(\abs{X}) \]
  in $D(\Lambda)$.
\end{prop}
\begin{proof}
  By construction, the left-hand triangle and outer composite in the following diagram commute:
  \[\begin{tikzcd}
    \Spc \ar[hookrightarrow]{r}\ar[swap]{rd}{\abs{-}}
    & \Stk \ar{r}{\Chom}\ar[swap]{d}{\abs{-}}
    & D(\Lambda)\\
    & \Anima \ar[swap]{ru}{\Chom^{\Anima}}
    & 
  \end{tikzcd}\]
  \thmref{thm:props}\itemref{item:ayog1} implies that $\Chom(-) : \Stk \to D(\Lambda)$ is also left Kan extended from $\Spc$.
  It follows that the right-hand triangle commutes if and only if the outer one does.
\end{proof}

\subsection{Relative homology}
\label{ssec:relhom}

If $S$ is a compact topological manifold, then may describe the complex of relative chains $\Chom(X_{/S}) = a_{S,*} f_!f^!(\Lambda_S)$ in more ``absolute'' terms:
\[ \Chom(X_{/S}) \simeq a_{S,!} f_!f^!(\omega_S\vb{-T_{S}}) \simeq \Chom(X)\vb{-T_S}. \]
More generally, using chains on pairs as in \thmref{thm:props}\itemref{item:loc}, we have:

\begin{lem}\label{lem:v0b01}
  Let $S$ be a topological manifold of dimension $d$ and $f : X \to S$ a topological Artin stack over $S$.
  There is a canonical isomorphism in $D(\Lambda)$
  \[
    \Chom(X_{/S})
    \simeq \lim_{K \sub S} \Chom(X, X\setminus X_K)\vb{-T_S},
  \]
  where the limit is over compact subsets $K \sub S$ and $X_K = X\fibprod_S K \sub X$.\footnote{%
    Since $S$ is Hausdorff (by convention), $K \hook S$ is a closed immersion.
    Hence $X_K \hook X$ is a closed immersion by base change, and $X \setminus X_K$ is the complementary open immersion as in Footnote~\ref{fn:complement}.
  }
\end{lem}
\begin{proof}
  We have $\Chom(X_{/S}) \simeq a_{S,*} f_!f^!a_S^!\Lambda\vb{-T_S}$.
  Using the equivalence $\Shv(S) \to \lim_{K\sub S} \Shv(K)$ (via $*$-inverse image) and base change for the squares
  \[\begin{tikzcd}
    X_K \ar{r}{i_X}\ar{d}{f_K}
    & X \ar{d}{f}
    \\
    K \ar{r}{i}
    & S,
  \end{tikzcd}\]
  we can write
  \[f_!f^!a_S^!\Lambda \simeq \lim_K i_*i^*f_!f^!a_S^!\Lambda \simeq \lim_K i_*f_{K,!}i_X^*f^!a_S^!\Lambda.\]
  Thus $a_{S,*}f_!f^!a_S^!\Lambda$ is the limit over $K$ of
  \[
    a_{S,*}i_*f_{K,!}i_X^*f^!a_S^!\Lambda
    \simeq a_{S,!}i_!f_{K,!}i_X^*f^!a_S^!\Lambda
    \simeq a_{X,!}i_{X,!}i_X^*a_X^!\Lambda
    \simeq \Chom(X, X\setminus X_K)
  \]
  using the fact that $a_K = a_S \circ i$ is proper and that $a_S \circ i \circ f_K = a_X \circ i_X$.
\end{proof}

\begin{rem}\label{rem:lim1}
  Assume for simplicity that $S$ is an oriented $C^\infty$-manifold of dimension $d$ (so that $\vb{-T_S} \simeq [-d]$).
  At the level of homology groups, \lemref{lem:v0b01} means that there are canonical surjections
  \[
    \H_n(X_{/S})
    \twoheadrightarrow \lim_{K \sub S} \H_{n+d}(X, X\setminus X_K)
  \]
  with generally nontrivial kernels
  \[ {\lim_{K}}^1 \H_{n+d+1}(X, X\setminus X_K). \]
  Compare \cite[Thm.~7.3]{SpanierDuality}.
\end{rem}

\section{Equivariant homology}

\subsection{Definition}

Let $G$ be a group acting on an Artin stack $X$.\footnote{%
  As before, ``stack'' can be read in the topological, complex analytic, or algebraic category as preferred, and the same goes for ``group'' (it does not mean ``discrete group'').
}
Consider the projection
\[ f : [X/G] \to BG := [\pt/G] \]
from the quotient stack to the classifying stack.\footnote{%
  See e.g. \cite[\S 4]{stacksncts} for the definition of group actions on higher stacks and their quotient stacks.
}
For every commutative ring $\Lambda$, the complexes of $G$-equivariant (Borel--Moore) chains on $X$, denoted by
\[
  \Chom^G(X)
  := \Chom([X/G]_{/BG}),
  \quad
  \Chom^{\BM,G}(X)
  := \Chom([X/G]_{/BG}),
\]
are the (derived) global sections of the objects
\begin{align*}
  \CHOM^G(X)
  &:= \CHOM([X/G]_{/BG})
  := f_!f^!(\Lambda_{BG}),\\
  \CHOM^{\BM,G}(X)
  &:= \CHOMBM([X/G]_{/BG})
  := f_*f^!(\Lambda_{BG}),
\end{align*}
of $\Shv(BG)$.

\subsubsection{}\label{sssec:rnqaivsq}

Since $a_{BG} : BG \to \pt$ is smooth with $\vb{-T_{BG}} \simeq \vb{T_G}$, we have $\Lambda_{BG} \simeq \omega_{BG}\vb{T_G}$ by Poincaré duality.
Thus
\[ a_{BG,*}f_*f^!(\Lambda_{BG}) \simeq a_{BG,*}f_*f^!a_{BG}^!(\Lambda)\vb{T_G}, \]
or in other words, $\Chom^{\BM,G}(X) \simeq \CBM([X/G])\vb{T_G}$.

\subsubsection{}

For homology, we get the less clean description
\[ \Chom^{G}(X) = a_{BG,*}f_!f^!(\Lambda_{BG}) \simeq a_{BG,*}f_!f^!a_{BG}^!(\Lambda)\vb{T_G}. \]
When $a_{BG} : BG \to \pt$ is not proper\footnote{%
  This amounts to compactness of $G$.
  For example, the complex analytification of a linear algebraic group will be compact if and only if it is finite.
}, this cannot be identified with $\Chom([X/G])\vb{T_G}$.

\subsection{The Borel construction}
\label{ssec:Borel}

We will now relate our definitions of equivariant (Borel--Moore) homology to a model for the Borel construction.
To avoid orientation issues, we will work in the complex analytic or algebraic categories for simplicity.

Say a space $A$ is $k$-acyclic if for every eventually coconnective complex $K \in \Shv(\pt)$, the unit $K \to a_*a^*(K)$ is $k$-coconnective, where $a : A \to \pt$ is the projection.
Let $(E_nG)_{n}$ be a sequential diagram of oriented $C^\infty$-manifolds with free $G$-action such that for every integer $k>0$, $E_nG$ is $k$-acyclic for sufficiently large $n$.

For example, if $G$ is a linear algebraic group, then the construction of \cite[Rem.~1.4]{TotaroChow} provides an example of such an $(E_nG)_n$ by \cite[Prop.~4.15]{equilisse}.\footnote{%
  See also \cite[\S 3.1]{BernsteinLunts} for a construction when $G$ is a Lie group with finitely many components.
}

Note that the freeness assumption means that the quotient stacks\footnote{%
The quotients can be formed in the $C^\infty$ category, as the ``associated topological stack'' functor is colimit-preserving.
Similarly, if $X$ and $G$ are algebraic or complex analytic, then the quotients could be taken equivalently in the corresponding category.
} $B_nG := [E_nG/G]$ have trivial stabilizers\footnote{%
  and are Hausdorff (even manifolds) when $G$ acts properly on $E_nG$ (e.g. if $G$ is compact).
}.
We denote by $q_n : B_nG \to BG$ the projections.

\begin{prop}\label{prop:Borel}
    For every $\sF \in \Shv(BG)$, the canonical morphism
    $\sF\to \lim_n q_{n,*}q_n^*(\sF)$
    is invertible.
\end{prop}

\begin{cor}
  Let $X$ be an Artin stack with $G$-action.
  Then the canonical maps in $\Shv(BG)$
  \begin{align*}
    \CHOM^{\BM,G}(X)
    &\to \lim_n \CHOMBM\big(X \fibprod^{G} E_nG\big)[-d_n+g]
    \\
    \CHOM^G(X)
    &\to \lim_n q_{n,*} \CHOM\big(X \fibprod^G E_nG_{/B_nG}\big)
  \end{align*}
  are invertible, where $d_n$ is the dimension of $B_nG$ and $g=\dim(G)$.
\end{cor}
\begin{proof}
  Applying \propref{prop:Borel} to $\sF = \CHOM^{\BM,G}(X) = f_*f^!(\Lambda_{BG})$, where $f : [X/G] \to BG$, yields
  \[f_*f^!(\Lambda_{BG})\simeq \lim_n q_{n,*} q_{n}^* f_* f^! (\Lambda_{BG}).\]
  Using the Poincaré duality isomorphism $q_{n}^* \simeq q_{n}^! [-d_n]$ and the base change formula, we may write
  \begin{equation*}
    q_{n,*} q_{n}^* f_* f^! (\Lambda_{BG})
    \simeq q_{n,*} f_{n,*} f_n^! q_n^! (\Lambda_{BG})[-d_n]
  \end{equation*}
  where $f_n : X\fibprod^G E_n \to B_nG$ is the base change of $f$.
  Translating back, this is
  \[\CHOM^{\BM,G} (X\fibprod^G E_nG)[-d_n] \simeq \CHOM^{\BM} (X\fibprod^G E_nG)[-d_n+g]\]
  by \sssecref{sssec:rnqaivsq}.

  Similarly, \propref{prop:Borel} expresses $\sF = \CHOM^G(X) = f_!f^!(\Lambda_{BG})$ as the limit over $n$ of
  \begin{align*}
    q_{n,*} q_{n}^* f_! f^! (\Lambda_{BG})
    \lim_n q_{n,*} f_{n,!} f_{n}^! q_n^! (\Lambda_{BG})[-d_n]
    \lim_n q_{n,*} f_{n,!} f_n^! (\Lambda_{B_nG}).
  \end{align*}
  Translating back, this is $\CHOM (X\fibprod^G E_nG_{/B_nG})$.
\end{proof}

The proof of \propref{prop:Borel} involves some analysis using the cohomological t-structure on $\Shv(X)$ (see Footnote~\ref{fn:cohtstr}).
Given a sequential diagram $\{f_n : A_n \to S\}_n$ of Artin stacks, we say that $\{f_n\}_n$ is \emph{$k$-pro-acyclic} if for every eventually coconnective $\sF \in \Shv(S)$, there exists an index $N$ such that the unit map $\sF \to f_{n,*}f_n^*(\sF)$ has $k$-coconnective fibre for all $n \ge N$.
We make the following three observations:
\begin{enumerate}
  \item 
  If $\{f_n\}_n$ is $k$-pro-acyclic, then the induced map $\sF \to \lim_n f_{n,*}f_n^*(\sF)$ has $k$-coconnective fibre.
  Indeed, coconnectivity is stable under limits.

  \item
  If $\{f_n\}_n$ is $k$-pro-acyclic, then any smooth base change is $k$-pro-acyclic.
  Indeed, smooth $*$-inverse image is t-exact and commutes with $*$-direct image (using Poincaré duality and the fact that $!$-inverse image commutes with $*$-direct image).

  \item
  If the base change of $\{f_n\}_n$ along a smooth surjection $S' \twoheadrightarrow S$ becomes $k$-pro-acyclic, then $\{f_n : A_n \to S\}_n$ was already pro-acyclic.
  Indeed, $*$-inverse image along smooth surjections is t-exact and conservative.
\end{enumerate}

\begin{proof}[Proof of \propref{prop:Borel}]
  It will suffice to show that $\{q_n\}_n$ is $k$-pro-acyclic for every $k$.
  By base change along the surjective submersion $\pt \twoheadrightarrow BG$, this follows from the assumption that $\{a_{E_nG} : E_nG \to \pt\}_n$ is $k$-pro-acyclic.
\end{proof}

Passing to hypercohomology groups, the isomorphism of \propref{prop:Borel} gives rise for every $i \in \Z$ to a surjective map
\[\H^i(BG, \sF) \xrightarrow{\sim} \H^i(BG, \lim_n q_{n,*}q_n^*\sF) \twoheadrightarrow \lim_n \H^i(BG, q_{n,*}q_n^*\sF)\]
whose kernel is a $\lim^1$ term.

\begin{lem}\label{lem:ludpxyiz}
  If $\sF \in \Shv(BG)$ is eventually coconnective, then for every integer $i\in\Z$, there exists a sufficiently large index $N$ such that
  \[ \H^i(BG, \sF) \to \H^i(BG, q_{n,*}q_{n}^*\sF) \simeq \H^i(B_nG, q_n^*\sF) \]
  is invertible for all $n\ge N$.
\end{lem}
\begin{proof}
  As explained in \cite[Rem.~4.10]{equilisse}, this is a consequence of the fact that $\{q_n\}_n$ is $k$-pro-acyclic for all $k$.
\end{proof}

\begin{cor}\label{cor:HiGX}
  Let $X$ be an Artin stack with $G$-action.
  Then for every integer $i\in\Z$ there exists a sufficiently large index $N$ such that the map
  \begin{equation*}
    \H_i^G(X) \to \H_i\big(B_nG, f_{n,!}f_{n}^!(\Lambda_{BG})\big) \simeq \H_i(X \fibprod^G E_nG_{/B_nG})
  \end{equation*}
  is invertible for all $n\ge N$, where $f_{n} : X\fibprod^G E_nG \to B_nG$ is the base change of $f : [X/G] \to BG$.
  If the dualizing complex $\omega_X$ is eventually coconnective\footnote{%
    In the algebraic category, if $X$ is an Artin stack of finite dimension in the sense of \cite[Tag~0AFL]{Stacks}, then $\omega_X$ is eventually coconnective by \cite[Thm.~4.4]{equilisse}.
  }, then similarly the map
  \[
    \H_i^{\BM,G}(X)
    \to \H_{i+d_n-g}^{\BM}\big(X \fibprod^{G} E_nG\big)
  \]
  is invertible for all $n\ge N$.
\end{cor}
\begin{proof}
  To apply \lemref{lem:ludpxyiz}, we need to check that the sheaves $f_*f^!(\Lambda_{BG})$ and $f_!f^!(\Lambda_{BG})$ are eventually coconnective.
  Note that $f_!$, $f_*$ and $f^!$ all commute with smooth $*$-inverse image (see \sssecref{sssec:more base change} for the latter two).
  Since $*$-inverse image along $\pt \twoheadrightarrow BG$ is t-exact and conservative, it will thus suffice to check the eventually coconnectivity of $\CBM(X)$ and $\Chom(X)$ instead.

  For $\Chom(X)$, recall that by smooth codescent (\thmref{thm:props}\itemref{item:ayog1}), $\Chom(-)$ is left Kan extended from spaces.
  Since connectivity is stable under colimits, we reduce to the case where $X$ is a space, which is clear e.g. from the singular chain complex model of $\Chom(X)$.

  For $\CBM(X)$, the claim follows from the assumption that the dualizing complex $\omega_{X}$ is eventually coconnective, since formation of global sections is always left t-exact.
\end{proof}



\bibliographystyle{halphanum}

\noindent
Institute of Mathematics, Academia Sinica, Taipei 10617, Taiwan

\end{document}